\documentclass[10pt,reqno]{amsart}
\usepackage[hypertex]{hyperref}

\usepackage{amsmath,amsthm}
\usepackage{amssymb}
\usepackage{euscript}

\numberwithin{equation}{subsection}

\newcommand{\sqsp}{\renewcommand{\baselinestretch}{1.15}\tiny\normalsize}

\newcommand{\relphantom[1]}{\mathrel{\phantom{#1}}}

\newtheorem{theorem}[subsection]{Theorem}
\newtheorem{lemma}[subsection]{Lemma}
\newtheorem{proposition}[subsection]{Proposition}
\newtheorem{corollary}[subsection]{Corollary}

\theoremstyle{definition}
\newtheorem{definition}[subsection]{Definition}
\newtheorem{example}[subsection]{Example}

\newcommand{\bk}{\mathbf{k}}
\newcommand{\bR}{\mathbb{R}}
\newcommand{\rtwon}{\bR^{2n}}

\newcommand{\xtwo}{x_1 \otimes x_2}
\newcommand{\ytwo}{y_1 \otimes y_2}
\newcommand{\ztwo}{z_1 \otimes z_2}

\newcommand{\fg}{\mathfrak{g}}
\newcommand{\sltwo}{\mathfrak{sl}_2}

\newcommand{\muop}{\mu^{op}}
\newcommand{\mualpha}{\mu_\alpha}
\newcommand{\mubeta}{\mu_\beta}
\newcommand{\mun}{\mu^{(n)}}
\newcommand{\mualphabar}{\mu_{\alphabar}}
\newcommand{\muphi}{\mu_{\varphi^*}}

\newcommand{\bracebeta}{\{,\}_\beta}
\newcommand{\bracen}{\{,\}^{(n)}}
\newcommand{\bracephi}{\{,\}_{\varphi^*}}

\newcommand{\defn}{\buildrel \text{def} \over =}
\newcommand{\cinfty}{C^\infty}
\newcommand{\cinftym}{C^\infty(M)}
\newcommand{\cinftyn}{C^\infty(N)}
\newcommand{\hei}{\mathsf{H}}
\newcommand{\pthreeone}{\mathcal{P}_{3,1}(\zeta)}
\newcommand{\pthreeonezero}{\mathcal{P}_{3,1}(0)}
\newcommand{\pthreetwo}{\mathcal{P}_{3,2}}

\newcommand{\fphimatrix}{{\begin{pmatrix} f(\varphi^2(x)) & f(\varphi(x))\\ f(\varphi^3(x)) & f(\varphi^2(x))\end{pmatrix}}}

\newcommand{\alphahei}{{
\begin{pmatrix}
a_{11} & a_{12} & 0\\
a_{21} & a_{22} & 0\\
a_{31} & a_{32} & b
\end{pmatrix}}}

\newcommand{\alphaone}{{
\begin{pmatrix}
a_{11} & 0 & 0\\
0 & a_{22} & 0\\
a_{31} & a_{32} & a_{11}a_{22}
\end{pmatrix}
}}

\newcommand{\alphatwo}{{
\begin{pmatrix}
a_{11} & a_{12} & 0\\
0 & 0 & 0\\
a_{31} & a_{32} & 0
\end{pmatrix}
}}

\newcommand{\alphathree}{{
\begin{pmatrix}
0 & 0 & 0\\
a_{21} & a_{22} & 0\\
a_{31} & a_{32} & 0
\end{pmatrix}
}}

\newcommand{\alphafour}{{
\begin{pmatrix}
0 & 0 & 0\\
a_{21} & a_{22} & 0\\
a_{31} & a_{32} & 0
\end{pmatrix}
}}

\newcommand{\alphafive}{{
\begin{pmatrix}
a_{11} & 0 & 0\\
a_{21} & a_{11} & 0\\
a_{31} & a_{32} & a_{11}^2
\end{pmatrix}
}}

\newcommand{\diagonal}{{
\begin{pmatrix}
1/2 & 0 & \cdots & 0\\
0 & 1 & \cdots & 0\\
\vdots & \vdots & \ddots & \vdots\\
0 & 0 & \cdots & 1
\end{pmatrix}
}}

\newcommand{\alphamatrix}{{
\begin{pmatrix}
a_{11} & a_{12}/2 & \cdots & a_{1n}/2\\
2a_{21} & a_{22} & \cdots & a_{2n}\\
\vdots  & \vdots & \ddots & \vdots \\
2a_{n1} & a_{n2} & \cdots & a_{nn}
\end{pmatrix}
}}

\newcommand{\matrixa}{{
\begin{pmatrix}
a_{11} & a_{12}\\
a_{21} & a_{22}
\end{pmatrix}
}}

\newcommand{\matrixx}{{
\begin{pmatrix}
1 & 1\\
1 & 0
\end{pmatrix}
}}

\newcommand{\matrixxa}{{
\begin{pmatrix}
3 & 1/2\\
8 & 1
\end{pmatrix}
}}

\newcommand{\matrixxb}{{
\begin{pmatrix}
5 & 3/8\\
6 & 1/4
\end{pmatrix}
}}

\begin{document}

\title{Non-commutative Hom-Poisson algebras}
\author{Donald Yau}

\begin{abstract}
A Hom-type generalization of non-commutative Poisson algebras, called non-commutative Hom-Poisson algebras, are studied.  They are closed under twisting by suitable self-maps.  Hom-Poisson algebras, in which the Hom-associative product is commutative, are closed under tensor products.  Through (de)polarization Hom-Poisson algebras are equivalent to admissible Hom-Poisson algebras, each of which has only one binary operation.  Multiplicative admissible Hom-Poisson algebras are Hom-power associative.
\end{abstract}



\address{Department of Mathematics\\
    The Ohio State University at Newark\\
    1179 University Drive\\
    Newark, OH 43055, USA}
\email{dyau@math.ohio-state.edu}

\date{\today}
\maketitle

\sqsp

\section{Introduction}

A Poisson algebra $(A,\{,\},\mu)$ consists of a commutative associative algebra $(A,\mu)$ together with a Lie algebra structure $\{,\}$, satisfying the Leibniz identity:
\[
\{\mu(x,y),z\} = \mu(\{x,z\},y) + \mu(x,\{y,z\}).
\]
Poisson algebras are used in many fields in mathematics and physics.  In mathematics, Poisson algebras play a fundamental role in Poisson geometry \cite{vaisman}, quantum groups \cite{cp,dri87}, and deformation of commutative associative algebras \cite{ger2}.  In physics, Poisson algebras are a major part of deformation quantization \cite{kont}, Hamiltonian mechanics \cite{arnold}, and topological field theories \cite{ss}.  Poisson-like structures are also used in the study of vertex operator algebras \cite{fb}.

One way to generalize Poisson algebras is to omit the commutativity requirement.  Such a structure is called a non-commutative Poisson algebra.  When the associative product happens to be commutative, one has a Poisson algebra.  Some classification results of finite dimensional non-commutative Poisson algebras can be found in \cite{kubo}.  One can also think of a non-commutative Poisson algebra as a special case of a Leibniz pair \cite{fgv}.

The purpose of this paper is to study a twisted generalization of non-commutative Poisson algebras, called non-commutative Hom-Poisson algebras.  In a non-commutative Hom-Poisson algebra $A$, there is a linear self-map $\alpha$ (the twisting map) and two binary operations $\{,\}$ (the Hom-Lie bracket) and $\mu$ (the Hom-associative product).  The associativity, the Jacobi identity, and the Leibniz identity in a non-commutative Poisson algebra are replaced by their Hom-type (i.e., $\alpha$-twisted) analogues in a non-commutative Hom-Poisson algebra.  In particular, $(A,\mu,\alpha)$ is a Hom-associative algebra \cite{ms}, and $(A,\{,\},\alpha)$ is a Hom-Lie algebra \cite{hls}.  If the twisting map is the identity map, then a non-commutative Hom-Poisson algebra reduces to a non-commutative Poisson algebra.

Most of our results are about Hom-Poisson algebras, which are non-commutative Hom-Poisson algebras in which the Hom-associative products are commutative.  Hom-Poisson algebras were defined in \cite{ms3} by Makhlouf and Silvestrov.  It is shown in \cite{ms3} that Hom-Poisson algebras play the same role in the deformation of commutative Hom-associative algebras as Poisson algebras do in the deformation of commutative associative algebras.  Other Hom-type algebras are studied in \cite{ms,ms2} and \cite{yau0} - \cite{yau15} and the references therein.

The rest of this paper is organized as follows.  In section \ref{sec:basic} non-commutative Hom-Poisson algebras are defined.  It is shown that, just like Poisson algebras, Hom-Poisson algebras are closed under tensor products in a non-trivial way (Theorem \ref{thm:tensor}).  It should be noted that Hom-Lie algebras are not closed under tensor products in any non-trivial way.  In section \ref{sec:twist} it is shown that non-commutative Hom-Poisson algebras are closed under suitable twistings by weak morphisms (Theorem \ref{thm:twist}).  This is a unique feature for Hom-type algebras, as non-commutative Poisson algebras are not closed under such twistings.  Using a special case of this result, several classes of (non-commutative) Hom-Poisson algebras are constructed.

In section \ref{sec:admissible} it is shown that Hom-Poisson algebras are equivalent to admissible Hom-Poisson algebras, each of which has one twisting map and only one binary operation.  The correspondence between Hom-Poisson algebras and admissible Hom-Poisson algebras is the Hom-version of the correspondence between Poisson algebras and admissible Poisson algebras \cite{gr,mr}.  In section \ref{sec:power} it is shown that multiplicative admissible Hom-Poisson algebras are Hom-power associative.

\section{Basic properties of non-commutative Hom-Poisson algebras}
\label{sec:basic}

In this section, we introduce (non-commutative) Hom-Poisson algebras and study tensor products of Hom-Poisson algebras.

\subsection{Notations}

We work over a fixed field $\bk$ of characteristic $0$.  If $V$ is a $\bk$-module and $\mu \colon V^{\otimes 2} \to V$ is a bilinear map, then $\muop \colon V^{\otimes 2} \to V$ denotes the opposite map, i.e., $\muop = \mu\tau$, where $\tau \colon V^{\otimes 2} \to V^{\otimes 2}$ interchanges the two variables.  For a linear self-map $\alpha \colon V \to V$, denote by $\alpha^n$ the $n$-fold composition of $n$ copies of $\alpha$, with $\alpha^0 \equiv Id$.

Let us begin with the basic definitions regarding Hom-algebras.

\begin{definition}
\label{def:homalg}
\begin{enumerate}
\item
By a \textbf{Hom-module} we mean a pair $(A,\alpha)$ in which $A$ is a $\bk$-module and $\alpha \colon A \to A$ is a linear map, called the twisting map.
\item
By a \textbf{Hom-algebra} we mean a triple $(A,\mu,\alpha)$ in which $(A,\alpha)$ is a Hom-module and $\mu \colon A^{\otimes 2} \to A$ is a bilinear map, called the multiplication.  A Hom-algebra $(A,\mu,\alpha)$ and the corresponding Hom-module $(A,\alpha)$ are often abbreviated to $A$.
\item
A Hom-algebra $(A,\mu,\alpha)$ is said to be \textbf{multiplicative} if $\alpha\mu = \mu \alpha^{\otimes 2}$.  It is called \textbf{commutative} if $\mu = \muop$.
\end{enumerate}
\end{definition}

Unless otherwise specified, an algebra $(A,\mu)$ with $\mu \colon A^{\otimes 2} \to A$ is also regarded as a Hom-algebra $(A,\mu,Id)$ with identity twisting map.  Given a Hom-algebra $(A,\mu,\alpha)$, we often use the abbreviation
\[
\mu(x,y) = xy
\]
for $x,y \in A$.

Let us now recall the Hom-type generalizations of associative and Lie algebras from \cite{ms}.

\begin{definition}
\label{def:homass}
Let $(A,\mu,\alpha)$ be a Hom-algebra.
\begin{enumerate}
\item
The \textbf{Hom-associator} $as_A \colon A^{\otimes 3} \to A$ is defined as
\begin{equation}
\label{homassociator}
as_A = \mu (\mu \otimes \alpha - \alpha \otimes \mu).
\end{equation}
\item
The Hom-algebra $A$ is called a \textbf{Hom-associative algebra} if it satisfies the \textbf{Hom-associative identity}
\begin{equation}
\label{homassociativity}
as_A = 0.
\end{equation}
\item
The \textbf{Hom-Jacobian} $J_A \colon A^{\otimes 3} \to A$ is defined as
\begin{equation}
\label{homjacobian}
J_A = \mu (\mu \otimes \alpha)(Id + \sigma + \sigma^2),
\end{equation}
where $\sigma(x \otimes y \otimes z) = z \otimes x \otimes y$.  The sum $(Id + \sigma + \sigma^2)$ is called a \textbf{cyclic sum}.
\item
The Hom-algebra $A$ is called a \textbf{Hom-Lie algebra} if it satisfies the \textbf{Hom-Jacobi identity}
\begin{equation}
\label{homjacobi}
J_A = 0.
\end{equation}
\end{enumerate}
\end{definition}

When the twisting map $\alpha$ is the identity map, the above definitions reduce to the usual definitions of the associator, an associative algebra, the Jacobian, and a Lie algebra.  Examples of Hom-associative and Hom-Lie algebras can be found in \cite{ms,yau0,yau1}.  In terms of elements $x,y,z \in A$, the Hom-associator and the Hom-Jacobian are
\[
\begin{split}
as_A(x,y,z) &= (xy)\alpha(z) - \alpha(x)(yz),\\
J_A(x,y,z) &= (xy)\alpha(z) + (zx)\alpha(y) + (yz)\alpha(x).
\end{split}
\]

Let us recall the definition of a non-commutative Poisson algebra \cite{kubo}.

\begin{definition}
\label{def:noncommpoisson}
A \textbf{non-commutative Poisson algebra} $(A,\{,\},\mu)$ consists of
\begin{enumerate}
\item
a Lie algebra $(A,\{,\})$ and
\item
an associative algebra $(A,\mu)$
\end{enumerate}
such that the Leibniz identity
\[
\{x,yz\} = \{x,y\}z + y\{x,z\}
\]
is satisfied for all $x,y,z \in A$.  A \textbf{Poisson algebra} is a non-commutative Poisson algebra $(A,\{,\},\mu)$ in which $\mu$ is commutative.  A \textbf{morphism} of non-commutative Poisson algebras is a linear map that is a morphism of the underlying Lie algebras and associative algebras.
\end{definition}

In a non-commutative Poisson algebra $(A,\{,\},\mu)$, the Lie bracket $\{,\}$ is called the Poisson bracket, and $\mu$ is called the associative product.  The Leibniz identity says that $\{x,-\}$ is a derivation with respect to the associative product.  It can be rewritten in element-free form as
\[
\{,\}(Id \otimes \mu) = \mu\left(\{,\} \otimes Id + (Id \otimes \{,\})(1~2)\right),
\]
where $(1~2)(x \otimes y \otimes z) = y \otimes x \otimes z$.

Hom-Poisson algebras are first introduced in \cite{ms2} by Makhlouf and Silvestrov.  We now define the Hom-type generalization of a non-commutative Poisson algebra.

\begin{definition}
\label{def:hompoisson}
A \textbf{non-commutative Hom-Poisson algebra} $(A,\{,\},\mu,\alpha)$ consists of
\begin{enumerate}
\item
a Hom-Lie algebra $(A,\{,\},\alpha)$ and
\item
a Hom-associative algebra $(A,\mu,\alpha)$
\end{enumerate}
such that the \textbf{Hom-Leibniz identity}
\begin{equation}
\label{homleibniz}
\{,\}(\alpha \otimes \mu) = \mu\left(\{,\} \otimes \alpha + (\alpha \otimes \{,\})(1 ~ 2)\right)
\end{equation}
is satisfied.  A \textbf{Hom-Poisson algebra} is a non-commutative Hom-Poisson algebra $(A,\{,\},\mu,\alpha)$ in which $\mu$ is commutative \cite{ms2}.
\end{definition}

In a non-commutative Hom-Poisson algebra $(A,\{,\},\mu,\alpha)$, the operations $\{,\}$ and $\mu$ are called the \textbf{Hom-Poisson bracket} and the \textbf{Hom-associative product}, respectively.  In terms of elements $x,y,z \in A$, the Hom-Leibniz identity says
\begin{equation}
\label{homleibniz'}
\{\alpha(x),yz\} = \{x,y\}\alpha(z) + \alpha(y)\{x,z\},
\end{equation}
where as usual $\mu(x,y)$ is abbreviated to $xy$.  By the anti-symmetry of the Hom-Poisson bracket $\{,\}$, the Hom-Leibniz identity is equivalent to
\[
\{xy,\alpha(z)\} = \{x,z\}\alpha(y) + \alpha(x)\{y,z\}.
\]
A (non-commutative) Poisson algebra is exactly a multiplicative (non-commutative) Hom-Poisson algebra with identity twisting map.

Let us now provide some basic properties of non-commutative Hom-Poisson algebras.  Every associative algebra $(A,\mu)$ has a non-commutative Poisson algebra structure in which the Poisson bracket is the commutator bracket.  The following result is the Hom-type analogue of this observation.

\begin{proposition}
\label{prop:homasspoisson}
Let $(A,\mu,\alpha)$ be a Hom-associative algebra.  Then
\[
A^- = (A,[,] = \mu - \muop,\mu,\alpha)
\]
is a non-commutative Hom-Poisson algebra.
\end{proposition}

\begin{proof}
It is proved in \cite{ms} (Proposition 1.6) that $(A,[,],\alpha)$ is a Hom-Lie algebra.  Indeed, the commutator bracket $[,]$ is obviously anti-symmetric.  One can write out all twelve terms (in terms of $\mu$) in the Hom-Jacobian of $A^-$ and observe that their sum is zero.  To check the Hom-Leibniz identity \eqref{homleibniz'} for $A^-$, we compute as follows:
\[
\begin{split}
[x,y]\alpha(z) & +  \alpha(y)[x,z] - [\alpha(x),yz]\\
&= (xy)\alpha(z) - (yx)\alpha(z) + \alpha(y)(xz) - \alpha(y)(zx) - \alpha(x)(yz) + (yz)\alpha(x)\\
&= as_A(x,y,z) - as_A(y,x,z) + as_A(y,z,x).
\end{split}
\]
Since $as_A = 0$, we conclude that $A^-$ satisfies the Hom-Leibniz identity.
\end{proof}


Next we study tensor products of Hom-Poisson algebras.  For motivation, recall that Lie algebras are not closed under tensor products.  The same is true for Hom-Lie algebras.  On the other hand, Poisson algebras are closed under tensor products.  In the rest of this section, we show that Hom-Poisson algebras are also closed under tensor products.

We need the following preliminary result.

\begin{lemma}
\label{lem:homass}
Let $(A,\mu,\alpha)$ be a commutative Hom-associative algebra. Then the expression $(xy)\alpha(z)$ is invariant under permutations of $x,y,z \in A$.  In other words, we have
\[
\mu(\mu \otimes \alpha) = \mu(\mu \otimes \alpha)\pi
\]
for all the permutations $\pi$ on three letters.
\end{lemma}

\begin{proof}
Pick $x,y,z \in A$.  Then
\[
\begin{split}
(xy)\alpha(z) &= (yx)\alpha(z)\\
&= \alpha(y)(xz)\\
&= (xz)\alpha(y).
\end{split}
\]
So the expression $(xy)\alpha(z)$ is invariant under the transpositions $(1~2)$ and $(2~3)$.  The proof is complete because these two transpositions generate the symmetric group on three letters.
\end{proof}

\begin{definition}
\label{def:poissonassociator}
Let $(A,\{,\},\mu,\alpha)$ be a quadruple in which $(A,\alpha)$ is a Hom-module and $\{,\}, \mu \colon A^{\otimes 2} \to A$ are bilinear operations.  Define its \textbf{Hom-associator} $as_A$ and \textbf{Hom-Jacobian} $J_A$ using $\mu$ and $\{,\}$, respectively, i.e.,
\[
\begin{split}
as_A &= \mu(\mu \otimes \alpha - \alpha \otimes \mu),\\
J_A &= \{,\}(\{,\} \otimes \alpha)(Id + \sigma + \sigma^2),
\end{split}
\]
where $\sigma(x \otimes y \otimes z) = z \otimes x \otimes y$.
\end{definition}

Now we are ready to prove that Hom-Poisson algebras are closed under tensor products.

\begin{theorem}
\label{thm:tensor}
Let $(A_i,\{,\}_i,\mu_i,\alpha_i)$ be Hom-Poisson algebras for $i = 1,2$, and let $A = A_1 \otimes A_2$.  Define the operations $\alpha \colon A \to A$ and $\mu, \{,\} \colon A^{\otimes 2} \to A$ by:
\[
\begin{split}
\alpha &= \alpha_1 \otimes \alpha_2,\\
\mu(\xtwo,\ytwo) &= \mu_1(x_1,y_1) \otimes \mu_2(x_2,y_2),\\
\{\xtwo,\ytwo\} &= \{x_1,y_1\}_1 \otimes \mu_2(x_2,y_2) + \mu_1(x_1,y_1) \otimes \{x_2,y_2\}_2
\end{split}
\]
for $x_i,y_i \in A_i$.  Then $(A,\{,\},\mu,\alpha)$ is a Hom-Poisson algebra.
\end{theorem}

\begin{proof}
That $(A,\mu,\alpha)$ is a commutative Hom-associative algebra follows from the commutativity and Hom-associativity of both $\mu_i$.  Also, the commutativity of the $\mu_i$ and the anti-symmetry of the $\{,\}_i$ imply the anti-symmetry of $\{,\}$.  It remains to prove the Hom-Jacobi identity and the Hom-Leibniz identity in $A$.

To simplify the typography, we abbreviate $\mu_1$, $\mu_2$, and $\mu$ using juxtaposition and drop the subscripts in $\{,\}_i$ and $\alpha_i$. Pick $x = \xtwo$, $y = \ytwo$, and $z = \ztwo \in A$.  Then
\[
\begin{split}
\{\{x,y\}, \alpha(z)\}
&= \{\{x_1,y_1\} \otimes x_2y_2, \alpha(z_1) \otimes \alpha(z_2)\} + \{x_1y_1 \otimes \{x_2,y_2\}, \alpha(z_1) \otimes \alpha(z_2)\}\\
&= \underbrace{\{\{x_1,y_1\},\alpha(z_1)\} \otimes (x_2y_2)\alpha(z_2)}_{a}
+ \underbrace{\{x_1,y_1\}\alpha(z_1) \otimes \{x_2y_2, \alpha(z_2)\}}_{b}\\
&\relphantom{} + \underbrace{\{x_1y_1,\alpha(z_1)\} \otimes \{x_2,y_2\}\alpha(z_2)}_{c} + \underbrace{(x_1y_1)\alpha(z_1) \otimes \{\{x_2,y_2\},\alpha(z_2)\}}_{d}
\end{split}
\]
The cyclic sum over $x,y,z$ of the term $a$ is $0$ by Lemma \ref{lem:homass} and the Hom-Jacobi identity in $A_1$.  Likewise, the cyclic sum over $x,y,z$ of the term $d$ is $0$ by Lemma \ref{lem:homass} and the Hom-Jacobi identity in $A_2$.  It follows that the Hom-Jacobian $J_A(x,y,z)$ consists of only the cyclic sum over $x,y,z$ of the terms $b$ and $c$.  Using the Hom-Leibniz identity in the $A_i$, the terms $b$ and $c$ become
\[
\begin{split}
b &= \{x_1,y_1\}\alpha(z_1) \otimes \{x_2y_2, \alpha(z_2)\}\\
&= \{x_1,y_1\}\alpha(z_1) \otimes \alpha(x_2)\{y_2,z_2\} + \{x_1,y_1\}\alpha(z_1) \otimes \{x_2,z_2\}\alpha(y_2),\\
c &= \{x_1y_1,\alpha(z_1)\} \otimes \{x_2,y_2\}\alpha(z_2)\\
&= \alpha(x_1)\{y_1,z_1\} \otimes \{x_2,y_2\}\alpha(z_2) + \{x_1,z_1\}\alpha(y_1) \otimes \{x_2,y_2\}\alpha(z_2).
\end{split}
\]
Therefore, the Hom-Jacobian of $A$ applied to $x \otimes y \otimes z$ is:
\[
\begin{split}
J_A(x,y,z) &=
\underbrace{\{x_1,y_1\}\alpha(z_1) \otimes \alpha(x_2)\{y_2,z_2\}}_{b_1}
 + \underbrace{\{x_1,y_1\}\alpha(z_1) \otimes \{x_2,z_2\}\alpha(y_2)}_{b_2}\\
& \relphantom{} + \underbrace{\alpha(x_1)\{y_1,z_1\} \otimes \{x_2,y_2\}\alpha(z_2)}_{c_1}
+ \underbrace{\{x_1,z_1\}\alpha(y_1) \otimes \{x_2,y_2\}\alpha(z_2)}_{c_2}\\
& \relphantom{} + \underbrace{\{z_1,x_1\}\alpha(y_1) \otimes \alpha(z_2)\{x_2,y_2\}}_{b_3}
+ \underbrace{\{z_1,x_1\}\alpha(y_1) \otimes \{z_2,y_2\}\alpha(x_2)}_{b_4}\\
& \relphantom{} + \underbrace{\alpha(z_1)\{x_1,y_1\} \otimes \{z_2,x_2\}\alpha(y_2)}_{c_3}
+ \underbrace{\{z_1,y_1\}\alpha(x_1) \otimes \{z_2,x_2\}\alpha(y_2)}_{c_4}\\
& \relphantom{} + \underbrace{\{y_1,z_1\}\alpha(x_1) \otimes \alpha(y_2)\{z_2,x_2\}}_{b_5}
+ \underbrace{\{y_1,z_1\}\alpha(x_1) \otimes \{y_2,x_2\}\alpha(z_2)}_{b_6}\\
& \relphantom{} + \underbrace{\alpha(y_1)\{z_1,x_1\} \otimes \{y_2,z_2\}\alpha(x_2)}_{c_5}
+ \underbrace{\{y_1,x_1\}\alpha(z_1) \otimes \{y_2,z_2\}\alpha(x_2)}_{c_6}.
\end{split}
\]
Using the anti-symmetry of $\{,\}_i$ and the commutativity of $\mu_i$, observe that the following sums are $0$: $b_1 + c_6$, $b_2 + c_3$, $b_3 + c_2$, $b_4 + c_5$, $b_5 + c_4$, and $b_6 + c_1$.  This shows that $(A,\{,\},\alpha)$ satisfies the Hom-Jacobi identity $J_A = 0$.

Finally, we check the Hom-Leibniz identity in $A$.  Using the Hom-associativity and the Hom-Leibniz identity in the $A_i$ and Lemma \ref{lem:homass}, we compute as follows:
\[
\begin{split}
\{xy,\alpha(z)\}
&= \{x_1y_1 \otimes x_2y_2, \alpha(z_1) \otimes \alpha(z_2)\}\\
&= \{x_1y_1,\alpha(z_1)\} \otimes (x_2y_2)\alpha(z_2) + (x_1y_1)\alpha(z_1) \otimes \{x_2y_2, \alpha(z_2)\}\\
&= \alpha(x_1)\{y_1,z_1\} \otimes \alpha(x_2)(y_2z_2) + \{x_1,z_1\}\alpha(y_1) \otimes (x_2z_2)\alpha(y_2)\\
&\relphantom{} + \alpha(x_1)(y_1z_1) \otimes \alpha(x_2)\{y_2,z_2\} + (x_1z_1)\alpha(y_1) \otimes \{x_2,z_2\}\alpha(y_2)\\
&= (\alpha(x_1) \otimes \alpha(x_2))\left(\{y_1,z_1\} \otimes y_2z_2 + y_1z_1 \otimes \{y_2,z_2\}\right)\\
&\relphantom{} + \left(\{x_1,z_1\} \otimes x_2z_2 + x_1z_1 \otimes \{x_2,z_2\}\right)(\alpha(y_1) \otimes \alpha(y_2))\\
&= \alpha(x)\{y,z\} + \{x,z\}\alpha(y).
\end{split}
\]
This shows that $A$ satisfies the Hom-Leibniz identity.
\end{proof}

Setting $\alpha_i = Id_{A_i}$ in Theorem \ref{thm:tensor}, we recover the following well-known result about Poisson algebras.

\begin{corollary}
\label{cor:tensor}
Let $(A_i,\{,\}_i,\mu_i)$ be Poisson algebras for $i = 1,2$, and let $A = A_1 \otimes A_2$.  Define the operations $\mu, \{,\} \colon A^{\otimes 2} \to A$ by:
\[
\begin{split}
\mu(\xtwo,\ytwo) &= \mu_1(x_1,y_1) \otimes \mu_2(x_2,y_2),\\
\{\xtwo,\ytwo\} &= \{x_1,y_1\}_1 \otimes \mu_2(x_2,y_2) + \mu_1(x_1,y_1) \otimes \{x_2,y_2\}_2
\end{split}
\]
for $x_i,y_i \in A_i$.  Then $(A,\{,\},\mu)$ is a Poisson algebra.
\end{corollary}

\section{Twistings of non-commutative Hom-Poisson algebras}
\label{sec:twist}

In this section, we first observe that non-commutative Hom-Poisson algebras are closed under twisting by suitable self-maps.  A special case of this observation is that non-commutative Poisson algebras give rise to multiplicative non-commutative Hom-Poisson algebras via twisting by self-morphisms (Corollary \ref{cor2:twist}).  To study whether these twistings of non-commutative Poisson algebras actually give rise to non-commutative Poisson algebras, we introduce the concept of \emph{rigidity} in Definition \ref{def:deform}.  Twistings and rigidity of some classes of Poisson algebras are then studied.

\begin{definition}
\label{def:morphism}
Let $(A,\{,\},\mu,\alpha)$ be a quadruple in which $(A,\alpha)$ is a Hom-module and $\{,\}, \mu \colon A^{\otimes 2} \to A$ are bilinear operations.
\begin{enumerate}
\item
$A$ is \textbf{multiplicative} if
\[
\alpha\{,\} = \{,\} \alpha^{\otimes 2} \quad\text{and}\quad \alpha\mu = \mu\alpha^{\otimes 2}.
\]
\item
Let $(B,\{,\}_B,\mu_B,\alpha_B)$ be another such quadruple. A \textbf{weak morphism} $f \colon A \to B$ is a linear map such that
\[
f\{,\} = \{,\}_B f^{\otimes 2} \quad\text{and}\quad
f\mu = \mu_B f^{\otimes 2}.
\]
A \textbf{morphism} $f \colon A \to B$ is a weak morphism such that $f\alpha = \alpha_Bf$.
\end{enumerate}
\end{definition}

Note that a quadruple $(A,\{,\},\mu,\alpha)$ is multiplicative if and only if the twisting map $\alpha \colon A \to A$ is a morphism.

The following result says that (non-commutative) Hom-Poisson algebras are closed under twisting by self-weak morphisms.

\begin{theorem}
\label{thm:twist}
Let $(A,\{,\},\mu,\alpha)$ be a (non-commutative) Hom-Poisson algebra and $\beta \colon A \to A$ be a weak morphism.  Then
\[
A_\beta = (A,\{,\}_\beta = \beta\{,\},\mubeta = \beta\mu,\beta\alpha)
\]
is also a (non-commutative) Hom-Poisson algebra.  Moreover, if $A$ is multiplicative and $\beta$ is a morphism, then $A_\beta$ is a multiplicative (non-commutative) Hom-Poisson algebra.
\end{theorem}

\begin{proof}
If $\mu$ is commutative, then clearly so is $\mubeta$.  The rest of the proof applies whether $\mu$ is commutative or not.

Since $\beta$ is compatible with $\mu$ and $\{,\}$, the Hom-associators and the Hom-Jacobians of $A$ and $A_\beta$ are related as:
\[
\begin{split}
\beta^2 as_A &= (\beta\mu)\left(\beta\mu \otimes \beta\alpha - \beta\alpha \otimes \beta\mu\right)\\
&= as_{A_\beta}
\end{split}
\]
and
\[
\begin{split}
\beta^2 J_A &= (\beta\{,\})(\beta\{,\} \otimes \beta\alpha)(Id + \sigma + \sigma^2)\\
&= J_{A_\beta}.
\end{split}
\]
This implies that $(A,\mubeta,\beta\alpha)$ is a Hom-associative algebra and that $(A,\{,\}_\beta,\beta\alpha)$ is a Hom-Lie algebra.  Likewise, applying $\beta^2$ to the Hom-Leibniz identity \eqref{homleibniz} in $A$, we obtain
\[
(\beta\{,\})(\beta\alpha \otimes \beta\mu) = (\beta\mu)\left(\beta\{,\} \otimes \beta\alpha + (\beta\alpha \otimes \beta\{,\})(1~2)\right),
\]
which is the Hom-Leibniz identity in $A_\beta$.

For the multiplicativity assertion, suppose $A$ is multiplicative and $\beta$ is a morphism.  Let $\nu$ denote either $\mu$ or $\{,\}$.  Then
\[
\begin{split}
(\beta\alpha)\nu_\beta
&=(\beta\alpha)(\beta\nu)\\
&= \beta\nu\alpha^{\otimes 2} \beta^{\otimes 2}\\
&= \nu_\beta(\alpha\beta)^{\otimes 2}\\
&= \nu_\beta(\beta\alpha)^{\otimes 2}.
\end{split}
\]
This shows that $A_\beta$ is multiplicative.
\end{proof}

Two special cases of Theorem \ref{thm:twist} follow.  The following result says that each multiplicative (non-commutative) Hom-Poisson algebra gives rise to a derived sequence of multiplicative (non-commutative) Hom-Poisson algebras.

\begin{corollary}
\label{cor1:twist}
Let $(A,\{,\},\mu,\alpha)$ be a multiplicative (non-commutative) Hom-Poisson algebra.  Then
\[
A^n = \left(A,\bracen = \alpha^{n}\{,\},\mun = \alpha^{n}\mu, \alpha^{n+1}\right)
\]
is a multiplicative (non-commutative) Hom-Poisson algebra for each integer $n \geq 0$.
\end{corollary}

\begin{proof}
The multiplicativity of $A$ implies that $\alpha^n \colon A \to A$ is a morphism.  By Theorem \ref{thm:twist} $A_{\alpha^n} = A^n$ is a multiplicative (non-commutative) Hom-Poisson algebra.
\end{proof}

The following observation is the $\alpha = Id$ special case of Theorem \ref{thm:twist}.  It gives a twisting construction of multiplicative (non-commutative) Hom-Poisson algebras from (non-commutative) Poisson algebras.  A result of this form was first given by the author in \cite{yau1} for $G$-Hom-associative algebras.  This twisting construction highlights the fact that the usual category of (non-commutative) Poisson algebras is \emph{not} closed under twisting by self-morphisms, which, in view of Theorem \ref{thm:twist}, is in strong contrast with the category of (non-commutative) Hom-Poisson algebras.  This is a major conceptual difference between (non-commutative) Hom-Poisson algebras and (non-commutative) Poisson algebras.

\begin{corollary}
\label{cor2:twist}
Let $(A,\{,\},\mu)$ be a (non-commutative) Poisson algebra and $\beta \colon A \to A$ be a morphism.  Then
\[
A_\beta = \left(A,\bracebeta = \beta\{,\}, \mubeta = \beta\mu, \beta\right)
\]
is a multiplicative (non-commutative) Hom-Poisson algebra.
\end{corollary}

In the setting of Corollary \ref{cor2:twist}, two natural questions arise: Is the triple
\begin{equation}
\label{betadeformation}
A_\beta' = \left(A,\bracebeta = \beta\{,\}, \mubeta = \beta\mu\right)
\end{equation}
a (non-commutative) Poisson algebra?  If so, is it isomorphic to $A$ itself?  To study these questions, we introduce the following concepts.

\begin{definition}
\label{def:deform}
Let $(A,\{,\},\mu)$ be a non-commutative Poisson algebra.
\begin{enumerate}
\item
Given a morphism $\beta \colon A \to A$, the triple $A_\beta'$ in \eqref{betadeformation} is called the \textbf{$\beta$-twisting of $A$.}  A \textbf{twisting of $A$} is a $\beta$-twisting of $A$ for some morphism $\beta \colon A \to A$.
\item
The $\beta$-twisting $A_\beta'$ of $A$ is called \textbf{trivial} if
\[
\bracebeta = 0 = \mubeta.
\]
$A_\beta'$ is called \textbf{non-trivial} if either $\bracebeta \not= 0$ or $\mubeta \not= 0$.
\item
$A$ is called \textbf{rigid} if every twisting of $A$ is either trivial or isomorphic to $A$.
\end{enumerate}
\end{definition}

The reader is cautioned not to confuse the above notion of rigidity with Gerstenhaber's \cite{ger2}.

Every non-commutative Poisson algebra $A$ has at least one trivial twisting $A_\beta'$, which is obtained by taking $\beta = 0$.  As we will see below, it is possible for a $\beta$-twisting to be trivial even if $\beta$ is not the zero map.  On the other hand, if a non-commutative Poisson algebra $A$ has either a non-zero Lie bracket or a non-zero associative product, then the $Id_A$-twisting is non-trivial.  As we will see below, it is possible for a $\beta$-twisting to be isomorphic to $A$ even if $\beta$ is not the identity map.

The following result gives some basic criteria for non-rigidity.

\begin{proposition}
\label{prop:nonrigidity}
Let $(A,\{,\},\mu)$ be a non-commutative Poisson algebra.  Suppose there exists a morphism $\beta \colon A \to A$ such that either:
\begin{enumerate}
\item
$\mubeta = \beta\mu$ is not associative or
\item
$\{,\}_\beta = \beta\{,\}$ does not satisfy the Jacobi identity.
\end{enumerate}
Then $A$ is not rigid.
\end{proposition}

\begin{proof}
The $\beta$-twisting $A_\beta'$ is non-trivial, since otherwise $\mubeta$ would be associative and $\{,\}_\beta$ would satisfy the Jacobi identity.  For the same reason, the $\beta$-twisting $A_\beta'$ cannot be isomorphic to $A$.
\end{proof}

Using Proposition \ref{prop:nonrigidity} we now consider some classes of non-commutative Poisson algebras and study their (non-)rigidity.

\begin{example}[\textbf{Free associative algebras are not rigid}]
\label{ex:poly}
Let $A = (\bk(S),\mu)$ be the free unital associative algebra on a non-empty set $S$.  Equip $A$ with the non-commutative Poisson algebra structure in which the Poisson bracket is the commutator bracket (Proposition \ref{prop:homasspoisson}).  We claim that $A$ is not rigid.  By Proposition \ref{prop:nonrigidity} it suffices to show that $\mualpha = \alpha\mu$ is not associative for some morphism $\alpha$ on $A$.  Pick an element $X \in S$, and let $\alpha \colon A \to A$ be the morphism determined for $Y \in S$ by
\[
\alpha(Y) =
\begin{cases}
1 + X &\text{ if $Y = X$},\\
Y & \text{ if $Y \not= X$}.
\end{cases}
\]
For $n \geq 1$, we have
\[
\alpha^n(X) = n + X.
\]
The associator of $\mualpha$ applied to $(X,X,\alpha(X))$ is:
\[
\begin{split}
\mualpha(\mualpha(X,X),\alpha(X)) & - \mualpha(X,\mualpha(X,\alpha(X)))\\
&= \left(\alpha^2(X)\right)^3 - \alpha(X)\alpha^2(X)\alpha^3(X)\\
&= (2 + X)^3 - (1+X)(2+X)(3+X)\\
&\not= 0.
\end{split}
\]
So $\mualpha$ is not associative, and $A$ is not rigid.
\qed
\end{example}

\begin{example}[\textbf{Matrix algebras are not rigid}]
\label{ex:matrix}
Let $A = (M_n(\bk),\mu)$ be the associative algebra of $n \times n$ matrices over $\bk$ for some $n \geq 2$.  Equip $A$ with the non-commutative Poisson algebra structure in which the Poisson bracket is the commutator bracket (Proposition \ref{prop:homasspoisson}).  We claim that $A$ is not rigid.  By Proposition \ref{prop:nonrigidity} it suffices to show that $\mualpha = \alpha\mu$ is not associative for some morphism $\alpha$ on $A$.

To construct such a morphism, consider the diagonal matrix
\[
D = \diagonal.
\]
The inverse of $D$ is the same as $D$ except that its $(1,1)$-entry is $2$.  There is a morphism $\alpha \colon A \to A$ given by
\[
\alpha((a_{ij})) = D(a_{ij})D^{-1} = \alphamatrix.
\]
From here on, a matrix $(a_{ij}) \in A$ with $a_{ij} = 0$ whenever $i \geq 3$ or $j \geq 3$ is abbreviated to its upper left $2 \times 2$ submatrix
\[
\matrixa.
\]
Consider such a matrix
\[
X = \matrixx \in A.
\]
Then a quick computation shows that
\[
\mualpha(\mualpha(X,X),\alpha(X)) = \alpha^2(X^3) = \matrixxa,
\]
whereas
\[
\mualpha(X,\mualpha(X,\alpha(X))) = \alpha\left(X\alpha(X)\alpha^2(X)\right) = \matrixxb.
\]
So $\mualpha$ is not associative, and $A$ is not rigid.
\qed
\end{example}

\begin{example}[\textbf{Hom-Poisson algebras from linear Poisson structures}]
\label{ex:linear}
Let $\fg$ be a finite-dimensional Lie algebra, and let $(S(\fg),\mu)$ be its symmetric algebra.  If $\{e_i\}_{i=1}^n$ is a basis of $\fg$, then $S(\fg)$ is the polynomial algebra $\bk[e_1,\ldots,e_n]$.  Suppose the structure constants for $\fg$ are given by
\[
[e_i,e_j] = \sum_{k=1}^n c_{ij}^k e_k.
\]
Then the symmetric algebra $S(\fg)$ becomes a Poisson algebra with the Poisson bracket
\begin{equation}
\label{FG}
\{F,G\} = \frac{1}{2} \sum_{i,j,k=1}^n c_{ij}^k e_k \left(\frac{\partial F}{\partial e_i}\frac{\partial G}{\partial e_j} - \frac{\partial F}{\partial e_j}\frac{\partial G}{\partial e_i}\right)
\end{equation}
for $F,G \in S(\fg)$.  This Poisson algebra structure on $S(\fg)$ is called the \textbf{linear Poisson structure}.  Note that $\{e_i,e_j\} = [e_i,e_j]$.

Let $\alpha \colon \fg \to \fg$ be a Lie algebra morphism.  It extends to an associative algebra morphism $\alpha \colon S(\fg) \to S(\fg)$ determined by
\[
\alpha(F(e_1,\ldots,e_n)) = F(\alpha(e_1),\ldots,\alpha(e_n)).
\]
Moreover, we claim that $\alpha$ respects the Poisson bracket \eqref{FG}, i.e., $\alpha$ is a morphism of Poisson algebras on $S(\fg)$.  Indeed, suppose that
\[
\alpha\{F,G\} = \{\alpha(F),\alpha(G)\} \quad\text{and}\quad
\alpha\{H,G\} = \{\alpha(H),\alpha(G)\}
\]
for some $F,G,H \in S(\fg)$.  Then the Leibniz identity
\[
\{FH,G\} = \{F,G\}H + F\{H,G\}
\]
implies
\[
\begin{split}
\alpha\{FH,G\}
&= \alpha(\{F,G\})\alpha(H) + \alpha(F)\alpha(\{H,G\})\\
&= \{\alpha(F),\alpha(G)\}\alpha(H) + \alpha(F)\{\alpha(H),\alpha(G)\}\\
&= \{\alpha(F)\alpha(H),\alpha(G)\}\\
&= \{\alpha(FH),\alpha(G)\}.
\end{split}
\]
Likewise, if
\[
\alpha\{F,G\} = \{\alpha(F),\alpha(G)\} \quad\text{and}\quad
\alpha\{F,H\} = \{\alpha(F),\alpha(H)\},
\]
then the Leibniz identity
\[
\{F,GH\} = \{F,G\}H + G\{F,H\}
\]
implies
\[
\alpha\{F,GH\} = \{\alpha(F),\alpha(GH)\}.
\]
Therefore, since $S(\fg)$ is the polynomial algebra $\bk[e_1,\ldots,e_n]$, to check that
\[
\alpha\{F,G\} = \{\alpha(F),\alpha(G)\}
\]
for all $F,G \in S(\fg)$, it suffices to show
\begin{equation}
\label{eiej}
\alpha\{e_i,e_j\} = \{\alpha(e_i),\alpha(e_j)\}
\end{equation}
for all $i,j \in \{1,\ldots,n\}$.  The desired identity \eqref{eiej} is true because $\{e_i,e_j\} = [e_i,e_j]$ and $\alpha$ is a Lie algebra morphism on $\fg$.

Therefore, given a Lie algebra morphism $\alpha \colon \fg \to \fg$, the extended map $\alpha \colon S(\fg) \to S(\fg)$ is a morphism of Poisson algebras.  By Corollary \ref{cor2:twist} there is a multiplicative Hom-Poisson algebra
\[
S(\fg)_\alpha = (S(\fg),\{,\}_\alpha = \alpha\{,\}, \mualpha = \alpha\mu, \alpha),
\]
which reduces to the linear Poisson structure $S(\fg)$ if $\alpha = Id$.
\qed
\end{example}

\begin{example}[\textbf{$S(\sltwo)$ is not rigid}]
In this example, we show that the symmetric algebra $(S(\sltwo),\mu)$ on the Lie algebra $\sltwo$, equipped with the linear Poisson structure \eqref{FG}, is not rigid in the sense of Definition \ref{def:deform}.

The Lie algebra $\sltwo$ has a basis $\{e,f,h\}$, with respect to which the Lie bracket is given by
\[
[h,e] = 2e,\quad [h,f] = -2f, \quad [e,f] = h.
\]
To show that $S(\sltwo) = \bk[e,f,h]$ is not rigid, consider the Lie algebra morphism $\alpha \colon \sltwo \to \sltwo$ given by
\[
\alpha(e) = \lambda e,\quad \alpha(f) = \lambda^{-1}f, \quad \alpha(h) = h,
\]
where $\lambda \in \bk$ is a fixed scalar with $\lambda \not= 0,1$.  As in Example \ref{ex:linear}, denote by $\alpha \colon S(\sltwo) \to S(\sltwo)$ the extended map, which is a Poisson algebra morphism.  By Proposition \ref{prop:nonrigidity}, the Poisson algebra $S(\sltwo)$ is not rigid if $\mualpha = \alpha\mu$ is not associative.  We have
\[
\begin{split}
\mualpha(\mualpha(e,h),h) - \mualpha(e,\mualpha(h,h))
&= \alpha^2(e)\alpha^2(h)\alpha(h) - \alpha(e)\alpha^2(h)\alpha^2(h)\\
&= (\lambda^2 - \lambda)eh^2,
\end{split}
\]
which is not $0$ in the symmetric algebra $S(\sltwo)$ because $\lambda \not= 0,1$.  Therefore, $\mualpha$ is not associative, and the linear Poisson structure on $S(\sltwo)$ is not rigid.
\qed
\end{example}

\begin{example}[\textbf{Hom-Poisson algebras from Poisson manifolds}]
\label{ex:manifold}
In this example, we discuss how multiplicative Hom-Poisson algebras arise from Poisson manifolds, which include symplectic manifolds and Poisson-Lie groups.  The reader is referred to, e.g., \cite{cp,vaisman} for discussion of Poisson manifolds.  The ground field in this example is the field of real numbers.

Let $M$ be a Poisson manifold.  This means that $M$ is a smooth manifold and that the commutative associative algebra $(\cinftym,\mu)$ (under point-wise multiplication) of smooth $\bR$-valued functions on $M$ is equipped with a Poisson algebra structure
\[
\{,\} \colon \cinftym \times \cinftym \to \cinftym.
\]
Let $N$ be another Poisson manifold.  A \textbf{Poisson map} $\varphi \colon M \to N$ between two Poisson manifolds is a smooth map such that
\[
\{f,g\}\varphi = \{f\varphi, g\varphi\}
\]
for all $f,g \in \cinftyn$.  Given such a Poisson map, the induced map
\[
\varphi^* \colon \cinftyn \to \cinftym,\quad \varphi^*(f) = f\varphi
\]
is a morphism of Poisson algebras.

Let $\varphi \colon M \to M$ be a Poisson map.  By Corollary \ref{cor2:twist} the morphism $\varphi^* \colon \cinftym \to \cinftym$ yields a multiplicative Hom-Poisson algebra
\[
\cinftym_{\varphi^*} = \left(\cinftym, \bracephi = \varphi^*\{,\}, \muphi = \varphi^*\mu, \varphi^*\right).
\]
If $\varphi = Id_M$, then $\cinftym_{Id^*}$ is the original Poisson algebra $\cinftym$.
\qed
\end{example}

We now give sufficient conditions that guarantee that the Poisson algebra $\cinftym$ of smooth functions on $M$ is \emph{not} rigid in the sense of Definition \ref{def:deform}.

\begin{corollary}
\label{cor:rigidmanifold}
Let $M$ be a Poisson manifold.  Suppose there exist a Poisson map $\varphi \colon M \to M$, $x \in M$, and $f \in \cinftym$ such that the matrix
\[
\Phi = \fphimatrix
\]
has non-zero trace and non-zero determinant.  Then the Poisson algebra $\cinftym$ is not rigid.
\end{corollary}

\begin{proof}
Using Proposition \ref{prop:nonrigidity} and the notations in Example \ref{ex:manifold}, it suffices to show that $\muphi$ is not associative.  That the matrix $\Phi$ has non-vanishing trace and determinant means
\begin{equation}
\label{fphi}
f(\varphi^2(x)) \not= 0 \quad\text{and}\quad
\left(f(\varphi^2(x))\right)^2 \not= f(\varphi(x))f(\varphi^3(x)).
\end{equation}
Using the conditions in \eqref{fphi}, we have:
\[
\begin{split}
\muphi\left(\muphi(f,f),f\varphi\right)(x)
&= \left(f(\varphi^2(x))\right)^3\\
&\not= f(\varphi(x)) f(\varphi^2(x)) f(\varphi^3(x))\\
&= \muphi\left(f,\muphi(f,f\varphi)\right)(x).
\end{split}
\]
This shows that $\muphi$ is not associative, so $\cinftym$ is not rigid.
\end{proof}

The following example illustrates how the non-rigidity criterion in Corollary \ref{cor:rigidmanifold} can be applied.

\begin{example}[\textbf{$\cinfty(\rtwon)$ is not rigid}]
The Euclidean space $\rtwon$ is a Poisson manifold with Poisson structure
\[
\{f,g\} = \sum_{i,j=1}^n \left(\frac{\partial f}{\partial x_i}\frac{\partial g}{\partial x_{i+n}} - \frac{\partial f}{\partial x_{i+n}}\frac{\partial g}{\partial x_i}\right)
\]
for $f,g \in \cinfty(\rtwon)$.   Using Corollary \ref{cor:rigidmanifold} we show that the Poisson algebra $\cinfty(\rtwon)$ is not rigid in the sense of Definition \ref{def:deform}.

Consider the map $\varphi \colon \rtwon \to \rtwon$ defined as
\[
\varphi(a_1,\ldots,a_{2n}) = (a_1 + c_1, \ldots, a_{2n} + c_{2n}),
\]
where the $c_j \in \bR$ are fixed scalars, not all of which are zero, say, $c_i \not= 0$.  The map $\varphi$ is a Poisson map by Chain Rule.  Let $f \in \cinfty(\rtwon)$ be the function
\[
f(a_1,\ldots,a_{2n}) = a_i,
\]
and let $0$ be the origin in $\rtwon$.  For each $k \geq 0$, we have
\[
\varphi^k(0) = (kc_1, \ldots, kc_{2n}).
\]
So
\[
f(\varphi^2(0)) = 2c_i \not= 0
\]
and
\[
\begin{split}
f(\varphi^2(0))^2 - f(\varphi(0))f(\varphi^3(0))
&= (2c_i)^2 - c_i(3c_i)\\
&= c_i^2 \not= 0.
\end{split}
\]
Therefore, by Corollary \ref{cor:rigidmanifold} the Poisson algebra $\cinfty(\rtwon)$ is not rigid.
\qed
\end{example}

\begin{example}[\textbf{Every Poisson structure on the Heisenberg algebra is rigid}]
\label{ex:heisenberg}
Over the ground field of complex numbers, the Poisson algebra structures on the Heisenberg algebra are classified by Goze and Remm in \cite{gr} (section 2).  In this example, we show that these Poisson algebras are all rigid in the sense of Definition \ref{def:deform}.

Let us first recall the Goze-Remm classification of Poisson algebra structures on the Heisenberg algebra \cite{gr}.  Let $\hei$ be the complex Heisenberg algebra, which is the three-dimensional complex Lie algebra with a basis $\{X,Y,Z\}$ such that
\begin{equation}
\label{heisenberg}
[X,Y] = Z, \quad [X,Z] = 0 = [Y,Z].
\end{equation}
The isomorphism classes of Poisson algebra structures on $\hei$ are divided into two families.  First there is a one-parameter family of Poisson algebra structures on $\hei$,
\[
\pthreeone = \{XY = \zeta Z\},
\]
where $\zeta$ is any complex number.  The notation above means that in the Poisson algebra $\pthreeone$, the commutative associative product satisfies
\begin{equation}
\label{xyzetaz}
XY = \zeta Z = YX,
\end{equation}
and the unspecified binary products of basis elements are all zero.  The only other isomorphism class of Poisson algebra structure on $\hei$ is the Poisson algebra
\[
\pthreetwo = \{X^2 = Z\}.
\]
We aim to show that the Poisson algebras $\pthreeone$ and $\pthreetwo$ are all rigid in the sense of Definition \ref{def:deform}.

Consider the Lie algebra morphisms on the Heisenberg algebra. It follows from the Heisenberg algebra relations \eqref{heisenberg} that a linear map $\alpha \colon \hei \to \hei$ is a Lie algebra morphism if and only if its matrix with respect to the basis $\{X,Y,Z\}$ is
\begin{equation}
\label{alphaheisenberg}
\alpha = \alphahei \quad\text{with}\quad b = a_{11}a_{22} - a_{21}a_{12},
\end{equation}
where the $a_{ij}$ are arbitrary complex numbers.  If $\alpha$ is a morphism on either $\pthreeone$ or $\pthreetwo$, then we obtain further relations among the $a_{ij}$.

(i) We first show that $\pthreeone$ is rigid.  If $\alpha \colon \pthreeone \to \pthreeone$ is a Poisson algebra morphism, then applying $\alpha$ to the relations \eqref{xyzetaz} yields
\[
\zeta b = \zeta(a_{11}a_{22} + a_{21}a_{12}),
\]
so
\begin{equation}
\label{relation1}
\zeta a_{21}a_{12} = 0.
\end{equation}
Likewise, applying $\alpha$ to the relations
\[
X^2 = 0 = Y^2
\]
in $\pthreeone$ yields
\begin{equation}
\label{relation2}
\zeta a_{11}a_{21} = 0 = \zeta a_{12}a_{22}.
\end{equation}
Conversely, if a Lie algebra morphism $\alpha \colon \hei \to \hei$ satisfies \eqref{relation1} and \eqref{relation2}, then it is a Poisson algebra morphism on $\pthreeone$.  We consider the two cases: $\zeta = 0$ and $\zeta \not= 0$.

\begin{itemize}
\item
If $\zeta = 0$, then \eqref{relation1} and \eqref{relation2} do not impose further relations on $\alpha$ in  \eqref{alphaheisenberg}.  Since $\pthreeonezero$ has a trivial associative product $\mu$, every $\alpha$-twisting of $\pthreeonezero$ also has a trivial Hom-associative product $\mualpha$.  With $\alpha$ as in \eqref{alphaheisenberg}, the induced Hom-Lie bracket as in Corollary \ref{cor2:twist} is determined by
\[
[X,Y]_\alpha = bZ.
\]
If $b = 0$, then $[,]_\alpha = 0$ and the $\alpha$-twisting
\[
\pthreeonezero_\alpha' = (\pthreeonezero, [,]_\alpha, \mu_\alpha)
\]
of $\pthreeonezero$ is trivial.  If $b \not= 0$, then the map
\[
\pthreeonezero \to \pthreeonezero_\alpha'\quad\text{with}\quad
\begin{cases}
X & \mapsto~ X,\\
Y & \mapsto~ Y,\\
Z & \mapsto~ bZ
\end{cases}
\]
is an isomorphism of Poisson algebras.  In other words, the Poisson algebra $\pthreeonezero$ is rigid.

\item
If $\zeta \not= 0$, then the relations \eqref{relation1} and \eqref{relation2} become
\[
a_{21}a_{12} = a_{11}a_{21} = a_{12}a_{22} = 0.
\]
This means that $\alpha$ takes one of the following three forms:
\[
\begin{split}
\alpha_1 &= \alphaone,\\
\alpha_2 &= \alphatwo\quad\text{with}\quad a_{12}\not= 0,\\
\alpha_3 &= \alphathree\quad\text{with}\quad a_{21}\not= 0.
\end{split}
\]
For $\alpha = \alpha_2$ or $\alpha_3$, we have $\alpha(Z) = 0$, so $[,]_\alpha = 0 = \mualpha$.  This implies that the $\alpha_2$-twisting and the $\alpha_3$-twisting of $\pthreeone$ (with $\zeta \not= 0$) are both trivial.

For $\alpha = \alpha_1$, if $a_{11} = 0$ or $a_{22} = 0$, then $\alpha(Z) = 0$.  This implies that $[,]_\alpha = 0 = \mualpha$, so the $\alpha_1$-twisting of $\pthreeone$ (with $\zeta \not= 0$) is trivial when $a_{11}a_{22} = 0$.  On the other hand, suppose $a_{11}a_{22} \not= 0$.  Since
\[
\begin{split}
\mualpha(X,Y) &= \zeta a_{11}a_{22}Z = \mualpha(Y,X),\\
[X,Y]_\alpha &= a_{11}a_{22}Z,
\end{split}
\]
the map
\[
\pthreeone \to \pthreeone_\alpha' \quad\text{with}\quad
\begin{cases}
X & \mapsto~ X,\\
Y & \mapsto~ Y,\\
Z & \mapsto~ a_{11}a_{22}Z
\end{cases}
\]
is an isomorphism of Poisson algebras.  We have shown that when $\zeta \not= 0$, every twisting of $\pthreeone$ is either trivial or isomorphic to $\pthreeone$ itself.  Therefore, $\pthreeone$ is rigid.
\end{itemize}

(ii) Next we show that $\pthreetwo$ is rigid.  Suppose $\alpha \colon \pthreetwo \to \pthreetwo$ is a morphism of Poisson algebras.  Applying $\alpha$ to the relations
\[
Y^2 = 0\quad\text{and}\quad X^2 = Z,
\]
we obtain
\begin{equation}
\label{a12}
a_{12} = 0 \quad\text{and}\quad a_{11}a_{22} = a_{11}^2.
\end{equation}
Conversely, if a Lie algebra morphism $\alpha \colon \hei \to \hei$ satisfies \eqref{a12}, then it is a Poisson algebra morphism on $\pthreetwo$.  Therefore, the Poisson algebra morphisms $\alpha$ on $\pthreetwo$ are
\[
\alpha_4 = \alphafour
\]
and
\[
\alpha_5 = \alphafive \quad\text{with}\quad a_{11} \not= 0.
\]
Since $\alpha_4(Z) = 0$, we have $[,]_{\alpha_4} = 0 = \mu_{\alpha_4}$, so the $\alpha_4$-twisting of $\pthreetwo$ is trivial.

For $\alpha = \alpha_5$, we have $\alpha(Z) = a_{11}^2 Z$ and
\[
\mualpha(X,X) = a_{11}^2 Z  = [X,Y]_{\alpha}.
\]
So the map
\[
\pthreetwo \to (\pthreetwo)_{\alpha}' \quad\text{with}\quad
\begin{cases}
X & \mapsto~ X,\\
Y & \mapsto~ Y,\\
Z & \mapsto~ a_{11}^2Z
\end{cases}
\]
is an isomorphism of Poisson algebras.  We have shown that every twisting of $\pthreetwo$ is either trivial or isomorphic to $\pthreetwo$ itself.  Therefore, $\pthreetwo$ is rigid.
\qed
\end{example}

\section{Admissible Hom-Poisson algebras}
\label{sec:admissible}

A Poisson algebra has two binary operations, the Lie bracket and the commutative associative product.  It is shown in \cite{gr,mr} that Poisson algebras can be described using only one binary operation via the polarization-depolarization process.  The purpose of this section is to extend this alternative description of Poisson algebras to Hom-Poisson algebras.  In other words, we will show that a Hom-Poisson algebra can be described using only the twisting map and one binary operation.

We first define the Hom-algebras that correspond to Hom-Poisson algebras.

\begin{definition}
\label{def:admissible}
Let $(A,\mu,\alpha)$ be a Hom-algebra.  Then $A$ is called an \textbf{admissible Hom-Poisson algebra} if it satisfies
\begin{equation}
\label{admissibility}
as_A(x,y,z) = \frac{1}{3}\left\lbrace(xz)\alpha(y) - (zx)\alpha(y) + (yz)\alpha(x) - (yx)\alpha(z)\right\rbrace
\end{equation}
for all $x,y,z \in A$, where $as_A$ is the Hom-associator \eqref{homassociator} of $A$.
\end{definition}

As usual in \eqref{admissibility} the product $\mu$ is denoted by juxtapositions of elements in $A$.  An admissible Hom-Poisson algebra with $\alpha = Id$ is exactly an \textbf{admissible Poisson algebra} as defined in \cite{gr}.

To compare Hom-Poisson algebras and admissible Hom-Poisson algebras, we need the following function, which generalizes a  similar function in \cite{mr}.

\begin{definition}
Let $(A,\mu,\alpha)$ be a Hom-algebra.  Define the quadruple
\begin{equation}
\label{pa}
P(A) = \left(A, \{,\} = \frac{1}{2}(\mu - \muop), \bullet = \frac{1}{2}(\mu + \muop), \alpha\right),
\end{equation}
called the \textbf{polarization} of $A$.  We call $P$ the \textbf{polarization function}.
\end{definition}

The following result says that admissible Hom-Poisson algebras, and only these Hom-algebras, give rise to Hom-Poisson algebras via polarization.  It is the Hom-version of \cite{mr} (Example 2).

\begin{theorem}
\label{thm:polar}
Let $(A,\mu,\alpha)$ be a Hom-algebra.  Then the polarization $P(A)$ is a Hom-Poisson algebra if and only if $A$ is an admissible Hom-Poisson algebra.
\end{theorem}

The proof will be given below.  Assuming Theorem \ref{thm:polar} for the moment, first we observe that the polarization function is actually a bijection from admissible Hom-Poisson algebras to Hom-Poisson algebras.  To prove this statement, we introduce the following function.

\begin{definition}
Let $(A,\{,\},\bullet,\alpha)$ be a quadruple in which $(A,\alpha)$ is a Hom-module and $\{,\},\, \bullet \colon A^{\otimes 2} \to A$ are binary operations.  Define the Hom-algebra
\begin{equation}
\label{pminusa}
P^-(A) = \left(A,\mu = \{,\}+\bullet, \alpha\right),
\end{equation}
called the \textbf{depolarization} of $A$.  We call $P^-$ the \textbf{depolarization function}.
\end{definition}

The following result says that there is a bijective correspondence between admissible Hom-Poisson algebras and Hom-Poisson algebras via polarization and depolarization.  It is the Hom-version of \cite{gr} (Proposition 3).

\begin{corollary}
\label{cor:polar}
The polarization and the depolarization functions
\[
P \colon \{\text{admissible Hom-Poisson algebras}\} \rightleftarrows \{\text{Hom-Poisson algebras}\} \colon P^-
\]
are the inverses of each other.
\end{corollary}

\begin{proof}
If $(A,\mu,\alpha)$ is an admissible Hom-Poisson algebra, then $P(A)$ is a Hom-Poisson algebra by Theorem \ref{thm:polar}.  We have $P^-(P(A)) = A$ because
\[
\mu = \frac{1}{2}(\mu - \muop) + \frac{1}{2}(\mu + \muop).
\]
Conversely, suppose $(A,\{,\},\bullet,\alpha)$ is a Hom-Poisson algebra.  To see that $P^-(A)$ is an admissible Hom-Poisson algebra, note that the anti-symmetry of $\{,\}$ and the commutativity of $\bullet$ imply that
\[
\begin{split}
\{,\} &= \frac{1}{2}\left((\{,\} + \bullet) - (\{,\} + \bullet)^{op}\right),\\
\bullet &= \frac{1}{2}\left((\{,\} + \bullet) + (\{,\} + \bullet)^{op}\right).
\end{split}
\]
So the Hom-algebra $P^-(A)$ has the property that $P(P^-(A)) = A$, which is a Hom-Poisson algebra.  It follows from Theorem \ref{thm:polar} that $P^-(A)$ is an admissible Hom-Poisson algebra.  Since $P^-P$ and $PP^-$ are both identity functions, $P$ and $P^-$ are the inverses of each other.
\end{proof}

It should be noted that both the polarization and the depolarization functions preserve multiplicativity.  So the polarization of a multiplicative admissible Hom-Poisson algebra is a multiplicative Hom-Poisson algebra.  Conversely, the depolarization of a multiplicative Hom-Poisson algebra is a multiplicative admissible Hom-Poisson algebra.

We now prove Theorem \ref{thm:polar} with a series of Lemmas.  To prove the ``if" part of Theorem \ref{thm:polar}, we need some preliminary results.  The following observation says that admissible Hom-Poisson algebras are Hom-flexible.  It is the Hom-version of \cite{gr} (Proposition 4).  The notion of Hom-flexibility is a Hom-type generalization of the usual definition of flexibility and was first introduced in \cite{ms}.

\begin{lemma}
\label{lem:flexible}
Every admissible Hom-Poisson algebra $(A,\mu,\alpha)$ is Hom-flexible, i.e.,
\begin{equation}
\label{homflexible}
as_A(x,y,z) + as_A(z,y,x) = 0
\end{equation}
for all $x,y,z \in A$.
\end{lemma}

\begin{proof}
The required identity \eqref{homflexible} follows immediately from the defining identity \eqref{admissibility}, in which the right-hand side is anti-symmetric in $x$ and $z$.
\end{proof}

Next we observe that in an admissible Hom-Poisson algebra the cyclic sum of the Hom-associator is trivial.

\begin{lemma}
\label{lem:cyclichomass}
Let $(A,\mu,\alpha)$ be an admissible Hom-Poisson algebra.  Then
\begin{equation}
\label{S}
S_A(x,y,z) \defn as_A(x,y,z) + as_A(z,x,y) + as_A(y,z,x) = 0
\end{equation}
for all $x,y,z \in A$.
\end{lemma}

\begin{proof}
Using the defining identity \eqref{admissibility}, we have:
\[
\begin{split}
as_A(x,y,z)
&= \frac{1}{3}\left((yz)\alpha(x) + (xz)\alpha(y) - (yx)\alpha(z) - (zx)\alpha(y)\right)\\
&= -\frac{1}{3}\left((zy)\alpha(x) - (yz)\alpha(x) + (xy)\alpha(z) - (xz)\alpha(y)\right)\\
&\relphantom{} + \frac{1}{3}\left((xy)\alpha(z) - (yx)\alpha(z) + (zy)\alpha(x) - (zx)\alpha(y)\right)\\
&= - as_A(z,x,y) + as_A(x,z,y)\\
&= - as_A(z,x,y) - as_A(y,z,x).
\end{split}
\]
The last equality above follows from Hom-flexibility (Lemma \ref{lem:flexible}). Therefore, we conclude that $S_A = 0$.
\end{proof}

Next we show that the polarization of an admissible Hom-Poisson algebra is commutative Hom-associative.

\begin{lemma}
\label{lem:adhomass}
Let $(A,\mu,\alpha)$ be an admissible Hom-Poisson algebra.  Then
\[
\left(A,\bullet = \frac{1}{2}(\mu + \muop),\alpha\right)
\]
is a commutative Hom-associative algebra.
\end{lemma}

\begin{proof}
It is obvious that $\bullet = (\mu + \muop)/2$ is commutative. To show that the Hom-associator
\[
as_{P(A)} = \bullet (\bullet \otimes \alpha - \alpha \otimes \bullet)
\]is trivial, pick $x,y,z \in A$.  We write $\mu$ using juxtaposition of elements in $A$.  Expanding $as_{P(A)}$ in terms of $\mu$, we have:
\begin{equation}
\label{4as}
\begin{split}
4as_{P(A)}(x,y,z) &= (xy)\alpha(z) + (yx)\alpha(z) + \alpha(z)(xy) + \alpha(z)(yx)\\
&\relphantom{} - \alpha(x)(yz) - \alpha(x)(zy) - (yz)\alpha(x) - (zy)\alpha(x)\\
&= as_{A}(x,y,z) - as_{A}(z,y,x) + (yx)\alpha(z) - (yz)\alpha(x)\\
&\relphantom{} - as_A(z,x,y) + (zx)\alpha(y) + as_A(x,z,y) - (xz)\alpha(y)
\end{split}
\end{equation}
Using \eqref{admissibility} and Hom-flexibility (Lemma \ref{lem:flexible}), we can combine six of the eight terms above as follows:
\begin{equation}
\label{4as1}
\begin{split}
as_{A}(x,y,z) & - as_{A}(z,y,x) + (yx)\alpha(z) - (yz)\alpha(x) + (zx)\alpha(y) - (xz)\alpha(y)\\
&= as_A(x,y,z) + as_A(x,y,z) - 3as_A(x,y,z)\\
&= -as_A(x,y,z)
\end{split}
\end{equation}
By Lemmas \ref{lem:flexible} and \ref{lem:cyclichomass}, the other two terms in the last line in \eqref{4as} become:
\begin{equation}
\label{4as2}
\begin{split}
- as_A(z,x,y) + as_A(x,z,y)
&= - as_A(z,x,y) - as_A(y,z,x)\\
&= as_A(x,y,z).
\end{split}
\end{equation}
Using \eqref{4as1} and \eqref{4as2} in \eqref{4as}, we conclude that $as_{P(A)} = 0$.
\end{proof}

Now we observe that the polarization of an admissible Hom-Poisson algebra is a Hom-Lie algebra.

\begin{lemma}
\label{lem:adhomlie}
Let $(A,\mu,\alpha)$ be a Hom-algebra.  Then
\begin{equation}
\label{4j}
\begin{split}
4J_{P(A)}(x,y,z) &= as_A(x,y,z) + as_A(z,x,y) + as_A(y,z,x)\\
&\relphantom{} - as_A(y,x,z) - as_A(z,y,x) - as_A(x,z,y)
\end{split}
\end{equation}
for all $x,y,z \in A$, where $J_{P(A)}$ is the Hom-Jacobian \eqref{homjacobian} of the polarization $P(A)$ \eqref{pa}.  Moreover, if $A$ is an admissible Hom-Poisson algebra, then
\[
\left(A,\{,\} = \frac{1}{2}(\mu - \muop),\alpha\right)
\]
is a Hom-Lie algebra.
\end{lemma}

\begin{proof}
Since $\{,\} = (\mu - \muop)/2$, one can expand
\[
4J_{P(A)} = 4\{,\}(\{,\} \otimes \alpha)(Id + \sigma + \sigma^2)
\]
in terms of $\mu$.  The resulting twelve terms are then written in terms of the Hom-associator $as_A$.  The result is the right-hand side of \eqref{4j}.

For the second assertion, first note that $\{,\}$ is clearly anti-symmetric.  Next observe that the identity \eqref{4j} can be rewritten as
\[
4J_{P(A)}(x,y,z) = S_A(x,y,z) - S_A(y,x,z),
\]
where $S_A$ is the cyclic sum of the Hom-associator defined in \eqref{S}.  Since $S_A = 0$ in an admissible Hom-Poisson algebra (Lemma \ref{lem:cyclichomass}), we conclude that $J_{P(A)} = 0$.  In other words, $(A,\{,\},\alpha)$ satisfies the Hom-Jacobi identity.
\end{proof}

The following result says that the polarization of an admissible Hom-Poisson algebra satisfies the Hom-Leibniz identity \eqref{homleibniz'}.

\begin{lemma}
\label{lem:adhomleibniz}
Let $(A,\mu,\alpha)$ be a Hom-algebra.  Then the polarization $P(A)$ satisfies
\begin{equation}
\label{paleibniz}
\begin{split}
4 & \left(\{\alpha(x), y \bullet z\} - \{x,y\} \bullet \alpha(z) - \alpha(y) \bullet \{x,z\}\right)\\
&= as_A(x,y,z) + as_A(z,y,x) + as_A(x,z,y) + as_A(y,z,x)\\
&\relphantom{}  - as_A(y,x,z) - as_A(z,x,y)
\end{split}
\end{equation}
for all $x,y,z \in A$.  Moreover, if $A$ is an admissible Hom-Poisson algebra, then the polarization $P(A)$ satisfies the Hom-Leibniz identity.
\end{lemma}

\begin{proof}
Since $\{,\} = (\mu - \muop)/2$ and $\bullet = (\mu + \muop)/2$, the left-hand side of \eqref{paleibniz} can be expanded in terms of $\mu$ into twelve terms.  The result can be written in terms of the Hom-associator $as_A$, which turns out to be the right-hand side of \eqref{paleibniz}.

For the second assertion, suppose that $A$ is an admissible Hom-Poisson algebra.  Then Hom-flexibility (Lemma \ref{lem:flexible}) implies that the right-hand side of \eqref{paleibniz} is $0$.  We conclude that
\[
\{\alpha(x), y \bullet z\} = \{x,y\} \bullet \alpha(z) + \alpha(y) \bullet \{x,z\},
\]
which is the Hom-Leibniz identity in the polarization $P(A)$.
\end{proof}

Next we show that only admissible Hom-Poisson algebras can give rise to Hom-Poisson algebras via polarization.

\begin{lemma}
\label{papoisson}
Let $(A,\mu,\alpha)$ be a Hom-algebra such that the polarization $P(A)$ is a Hom-Poisson algebra.  Then $A$ is an admissible Hom-Poisson algebra.
\end{lemma}

\begin{proof}
We need to prove the identity \eqref{admissibility}.  Pick $x,y,z \in A$.  We will express the Hom-associator $as_A$ in several different forms and compare them.

On the one hand, the Hom-Jacobi identity $J_{P(A)} = 0$ and \eqref{4j} imply that
\begin{equation}
\label{xyz1}
\begin{split}
as_A(x,y,z) &= as_A(y,x,z) - as_A(y,z,x)\\
&\relphantom{} - as_A(z,x,y) + as_A(z,y,x) + as_A(x,z,y).
\end{split}
\end{equation}
Moreover, the Hom-Leibniz identity in $P(A)$ and \eqref{paleibniz} imply that
\begin{equation}
\label{xyz2}
\begin{split}
as_A(x,y,z) &= as_A(y,x,z) - as_A(y,z,x)\\
&\relphantom{} + as_A(z,x,y) - as_A(z,y,x) - as_A(x,z,y).
\end{split}
\end{equation}
Adding \eqref{xyz1} and \eqref{xyz2} and dividing the result by $2$, we obtain
\begin{equation}
\label{xyz3}
as_A(x,y,z) = as_A(y,x,z) - as_A(y,z,x),
\end{equation}
which we will use in a moment.

On the other hand, since $\mu = \{,\} + \bullet$, we can expand the Hom-associator $as_A$ in terms of $\{,\}$ and $\bullet$ as follows:
\begin{equation}
\label{xyz4}
\begin{split}
as_A(x,y,z)
&= (xy)\alpha(z) - \alpha(x)(yz)\\
&= \{\{x,y\},\alpha(z)\} + \{x \bullet y,\alpha(z)\} + \{x,y\} \bullet \alpha(z) + (x \bullet y) \bullet \alpha(z)\\
&\relphantom{} - \{\alpha(x),\{y,z\}\} - \{\alpha(x),y \bullet z\} - \alpha(x) \bullet \{y,z\} - \alpha(x) \bullet (y \bullet z)
\end{split}
\end{equation}
Since the polarization $P(A)$ is assumed to be a Hom-Poisson algebra, we have:
\begin{equation}
\label{xyz5}
\begin{split}
0 &= as_{P(A)}(x,y,z) = (x \bullet y) \bullet \alpha(z) - \alpha(x) \bullet (y \bullet z),\\
0 &= \{x,z\} \bullet \alpha(y) - \alpha(y) \bullet \{x,z\}\\
&= \{x \bullet y,\alpha(z)\} - \alpha(x) \bullet \{y,z\} - \{\alpha(x),y \bullet z\} + \{x,y\} \bullet \alpha(z),\\
\{\{x,z\},\alpha(y)\} &= \{\{x,y\},\alpha(z)\} - \{\alpha(x),\{y,z\}\}.
\end{split}
\end{equation}
Using the identities \eqref{xyz5} in \eqref{xyz4}, we obtain:
\[
\begin{split}
4as_A(x,y,z)
&= 4\{\{x,z\},\alpha(y)\}\\
&= (xz)\alpha(y) - (zx)\alpha(y) - \alpha(y)(xz) + \alpha(y)(zx)\\
&= (xz)\alpha(y) - (zx)\alpha(y) + as_A(y,x,z) - (yx)\alpha(z) - as_A(y,z,x) + (yz)\alpha(x)\\
&= (xz)\alpha(y) - (zx)\alpha(y) + (yz)\alpha(x) - (yx)\alpha(z) + as_A(x,y,z),
\end{split}
\]
where the last equality follows from \eqref{xyz3}.  Finally, subtracting $as_A(x,y,z)$ in the above calculation and dividing the result by $3$, we obtain the desired identity \eqref{admissibility}.
\end{proof}

\begin{proof}[Proof of Theorem \ref{thm:polar}]
If $A$ is an admissible Hom-Poisson algebra, then Lemmas \ref{lem:adhomass}, \ref{lem:adhomlie}, and \ref{lem:adhomleibniz} imply that the polarization $P(A)$ is a Hom-Poisson algebra.  The converse is Lemma \ref{papoisson}.
\end{proof}

\section{Hom-power associativity}
\label{sec:power}

The purpose of this section is to show that multiplicative admissible Hom-Poisson algebras are Hom-power associative.  Power associativity of admissible Poisson algebras is proved in \cite{gr} (Proposition 6).

Let us first recall the definition of a Hom-power associative algebra from \cite{yau15}.

\begin{definition}
\label{def:hompower}
Let $(A,\mu,\alpha)$ be a Hom-algebra, $x \in A$, and $n$ be a positive integer.
\begin{enumerate}
\item
The \textbf{$n$th Hom-power} $x^n \in A$ is defined inductively as
\begin{equation}
\label{hompower}
x^1 = x, \qquad
x^n = x^{n-1}\alpha^{n-2}(x)
\end{equation}
for $n \geq 2$.
\item
For positive integers $i$ and $j$, define
\begin{equation}
\label{xij}
x^{i,j} = \alpha^{j-1}(x^i) \alpha^{i-1}(x^j).
\end{equation}
\item
$A$ is called \textbf{$n$th Hom-power associative} if
\begin{equation}
\label{nhpa}
x^n = x^{n-i,i}
\end{equation}
for all $x \in A$ and $i \in \{1,\ldots, n-1\}$.
\item
$A$ is called \textbf{Hom-power associative} if $A$ is $n$th Hom-power associative for all $n \geq 2$.
\end{enumerate}
\end{definition}

By definition
\[
x^2 = xx, \quad x^n = x^{n-1,1}
\]
for all $n \geq 2$.

If the twisting map $\alpha$ is the identity map, then
\[
x^n = x^{n-1}x,\quad x^{i,j} = x^ix^j,
\]
and $n$th Hom-power associativity reduces to
\begin{equation}
\label{npa}
x^n = x^{n-i}x^i
\end{equation}
for all $x \in A$ and $i \in \{1,\ldots, n-1\}$.  Therefore,
Hom-powers and ($n$th) Hom-power associativity become Albert's right powers and ($n$th) power associativity \cite{albert1,albert2} if $\alpha = Id$.  Examples of Hom-power associative algebras include multiplicative right Hom-alternative algebras and non-commutative Hom-Jordan algebras.  Other results for Hom-power associative algebras can be found in \cite{yau15}.

By definition, power associativity  involves infinitely many defining identities, namely, \eqref{npa} for all $n$.  A well-known result of Albert \cite{albert1} says that an algebra $(A,\mu)$ is power associative if and only if it is third and fourth power associative, i.e., the condition \eqref{npa} holds for $n = 3$ and $4$.  Moreover, for \eqref{npa} to hold for $n = 3$ and $4$, it is necessary and sufficient that
\[
(xx)x = x(xx) \quad\text{and}\quad ((xx)x)x = (xx)(xx)
\]
for all $x \in A$.  The Hom-versions of these statements are also true.  More precisely, the author proved in \cite{yau15} that a multiplicative Hom-algebra $(A,\mu,\alpha)$ is Hom-power associative if and only if it is third and fourth Hom-power associative, which in turn is equivalent to
\begin{equation}
\label{34}
x^2\alpha(x) = \alpha(x)x^2 \quad\text{and}\quad
x^4 = \alpha(x^2)\alpha(x^2)
\end{equation}
for all $x \in A$.

The following result is the Hom-version of \cite{gr} (Proposition 6).

\begin{theorem}
\label{thm:power}
Every multiplicative admissible Hom-Poisson algebra is Hom-power associative.
\end{theorem}

\begin{proof}
As discussed above, by a result in \cite{yau15} it suffices to prove the two equalities in \eqref{34}.  Hom-flexibility (Lemma \ref{lem:flexible}) implies that
\[
0 = as_A(x,x,x) = x^2\alpha(x) - \alpha(x)x^2,
\]
which proves the first identity in \eqref{34}.  To prove the other equality in \eqref{34}, note that Hom-flexibility implies that:
\[
\begin{split}
0 &= as_A(\alpha(x),x^2,\alpha(x))\\
&= (\alpha(x)x^2)\alpha^2(x) - \alpha^2(x)(x^2\alpha(x)).
\end{split}
\]
Together with the first identity in \eqref{34}, we have:
\begin{equation}
\label{x4}
\begin{split}
x^4 & \defn (x^2\alpha(x))\alpha^2(x) = (\alpha(x)x^2)\alpha^2(x)\\
&= \alpha^2(x)(x^2\alpha(x)) = \alpha^2(x)(\alpha(x)x^2).
\end{split}
\end{equation}
Using multiplicativity and \eqref{x4}, the defining identity \eqref{admissibility} applied to $(\alpha(x),\alpha(x),x^2)$ says that:
\[
\begin{split}
0 &= 3as_A(\alpha(x),\alpha(x),x^2) - (\alpha(x)x^2)\alpha^2(x) + (x^2\alpha(x))\alpha^2(x)\\
&\relphantom{} - (\alpha(x)x^2)\alpha^2(x) + (\alpha(x)\alpha(x))\alpha(x^2)\\
&= 3\left\{(\alpha(x)\alpha(x))\alpha(x^2) - \alpha^2(x)(\alpha(x)x^2)\right\} - x^4 + (\alpha(x)\alpha(x))\alpha(x^2)\\
&= 4\alpha(x^2)\alpha(x^2) - 4x^4.
\end{split}
\]
We have proved the second identity in \eqref{34}.
\end{proof}


\end{document}